\documentclass{ws-ijb}
\usepackage{graphicx}
\usepackage{subfigure}
\usepackage{color}
\usepackage{multirow}
\newcommand{\ba}{\begin{array} }
\newcommand{\ea}{\end{array} }
\newcommand{\bae}{\begin{eqnarray}}
\newcommand{\eae}{\end{eqnarray}}
\newcommand{\bea}{\begin{eqnarray*}}
\newcommand{\eea}{\end{eqnarray*}}
\newcommand{\be}{\begin{equation}}
\newcommand{\ee}{\end{equation}}

\begin{document}

\markboth{Yun Kang and Dieter Armbruster }
{Noise and seasonal effects on the dynamics of plant-herbivore models with monotonic plant growth functions}

\catchline{}{}{}{}{}
%
%%%%%%%%%%%%%%%%%%%%%%%%%%%%%%%%%%%%%%%%%%%%%%%%%%%%%%%%%%%%%%%%%%%%

\title{Noise and seasonal effects on the dynamics of plant-herbivore models with monotonic plant growth functions}

\author{Yun Kang}

\address{Applied Sciences and Mathematics, Arizona State University, \\Mesa, AZ 85212, USA.\\ 
\email{yun.kang@asu.edu} }

\author{Dieter Armbruster}

\address{School of Mathematical and Statistical Sciences, Arizona State University,\\ Tempe, AZ 85287-1804, USA. \email{armbruster@asu.edu}}

\maketitle

\begin{history}
\received{(Day Mth. Year)}
\revised{(Day Mth. Year)}
\end{history}
\begin{abstract}
{\bfseries Abstract.}\quad We formulate general plant-herbivore interaction models with monotone plant growth functions (rates). We study the impact of monotone plant growth functions in general plant-herbivore models on their dynamics. Our study shows that all monotone plant growth models generate a unique interior equilibrium and they are uniform {persistent} under certain range of {parameters} values. However, if the attacking rate of herbivore is too small or the quantity of plant is not enough, then herbivore goes extinct. Moreover, these models lead to noise sensitive bursting which can be identified as a dynamical mechanism for almost periodic outbreaks of the herbivore infestation. Montone and non-monotone plant growth models are contrasted with respect to bistability and crises of chaotic attractors. \end{abstract}
\keywords{Monotone growth models; uniformly {persistent}; Neimark-Sacker bifurcation; heteroclinic bifurcation; periodic infestations; bistability; noise bursting; crisis of chaos.}

\section{Introduction}
Interactions between plants and herbivores have been studied by ecologists for many decades. One focus of research is the effects of herbivores on plant dynamics \cite{Crawley1983}. In contrast, there is strong ecological evidence indicating that the population dynamics of plants has an important effect on the plant-herbivore interactions. In this article, we investigate how plants with different population dynamics contribute to the interactions. Models for plant growth vary strongly \cite{Crawley1990}:
Table \ref{mon_t1} lists eight discrete-time models of plant population growth. The first seven models are introduced in the paper by Law and Watkinson \cite{Watkinson1987} without inter-specific competition. All models are seasonal (discrete) models of the form $P_{t+1}  =P_t f(P_t)$, where $P_t$ is the density of a plant in season $t$ and $f(.)$ the per capita growth rate. In the absence of intra-specific competition the latter is given by $f(0)$, i.e. $1+q$ in models 1-3 and $w$ in models 4-8. The equilibrium density of the plant is given by $K$. The parameter $c$ is the space per plant at which interference with neighbors becomes appreciable \cite{Watkinson1987}. The interpretation of the power {parameter}, $b$, depends on the model. Generally, these models fall into two classes, depending on whether $P f(P)$ is a monotone function of $P$ or not. Models 1-3 are unimodal, i.e., they have a single hump. They lead to complicated dynamics including period doubling, period windows and chaos \cite{Guckenheimer1979}. Models 4-8 are monotone, leading to much simpler dynamics. Model 8 has a growth function of Holling-Type III \cite{Real1977}.

\begin{table}[ht]\label{mon_t1}
\begin{center}
\caption{Growth models of plant population density}
\begin{tabular}{|c|c|c|c|c|} \hline
Model &$f(P)$ &{Number of parameters} &$f(0)$ &equilibrium\\
\hline
1& $1+q - \frac{q P}{K}$&2&$1+q$ &$K$\\
\hline
2  &$e^{\ln{(1+q)}[1- \frac{P}{K}]}$&2&$1+q$&$K$\\
\hline
3& $e^{\ln{(1+q)}[1- \ln{(1+P)}]}$&2&$1+q$&$K$\\
\hline
4&$\frac{w}{1+c P}$&2&$w$&$\frac{w-1}{c}$\\
\hline
5&$\frac{w}{1+P^{b}}$&2&$w$&$(w-1)^{\frac{1}{b}}$\\
\hline
6&$\frac{w}{(1+P)^{b}}$&2&$w$&$w^{\frac{1}{b}}-1$\\
\hline
7&$\frac{w}{1+c P^{b}}$&3&$w$&$\frac{w^{\frac{1}{b}}-1}{c}$\\
\hline
%8&$\frac{w P^{b-1}}{1+ P^{b}}$&2&\multirow{2}{*}{0} &positive roots of\\
                                                                                % & &&$w P^{b-1} = 1+ P^b$\\
\multirow{3}{*}{8}&&&&\\
                           &$\frac{w P^{b-1}}{1+ P^{b}}$&2&0 &positive roots of\\
                             &&&&$w P^{b-1} = 1+ P^b$\\
\hline
\end{tabular}
\end{center}
\end{table}
\noindent Notice that Model \#2 is the well known Ricker model \cite{Ricker1954} which is unimodal and usually written as
\bae\label{discreterick}
P_{t+1}&=& P_t e^{ r (1- \frac{P_t}{K})}
\eae
while Model \#4 is the Beverton-Holt model \cite{Beverton1957} usually written as
\bae\label{discretebh}
P_{t+1}&=&\frac{K P_t}{e^{-r}K + P_t (1- e^{-r})}
\eae
The dynamics of the Ricker model (\ref{discreterick}) has been well studied. It shows period-doubling, chaos and period windows. A plant-herbivore model with Ricer dynamics in plant has been studied in {\cite{Kang2008} (also} see similar models in \cite{Karen2007} and \cite{Kon2006}) showing many forms of complex dynamics.

\indent There are fair amount of literatures on seasonal (discrete) multi-species interaction or stage structure models (e.g., Abbott and Dwyer \cite{Karen2007}; Cushing \cite{Cushing2001}; Dhirasakdanon \cite{Dhirasakdanon2010}; Jang \cite{Jang2006}; Kang \emph{et al} \cite{Kang2008}; Kon \cite{Kon2006}; Roeger \cite{Roeger2010}; Salceanu \cite{Salceanu2009}; Tuda and Iwasa \cite{Tuda1998}), among which a few studies are related to discrete prey-predator (or host-parasite) interaction models (e.g., Jang \cite{Jang2006}; Kang \emph{et al} \cite{Kang2008}; Kon \cite{Kon2006}). Tuda and Iwasa {\cite{Tuda1998}} developed scramble-type and {contest-type} models to examine an evolutionary shift in the mode of competition among the bean weevils. Jang {\cite{Jang2006}} studied a discrete-time Beverton-Holt stock recruitment model with Allee effects. Kang \emph{et al} {\cite{Kang2008}} and Kon {\cite{Kon2006}} studied a discrete plant-herbivore (or host-parasite) interaction model with Ricker dynamics as the growth function of plant (or host in Kon \cite{Kon2006}). In this article, we investigate the impact of general monotone plant growth models on the dynamics of plant-herbivore interaction. Our study is different from others and our results are new. We show that all monotone plant growth models generate a unique interior equilibrium (Theorem \ref{intercept1}) and they are {uniformly persistent} (see related definitions in \cite{Smith_Thieme2005}) for certain range of {parameters} values (Theorem \ref{persistent}). If the attacking rate of herbivore is too small or the quantity of plant is not enough, then herbivore goes extinct (Theorem \ref{1th2}). In addition, our numerical simulations suggest that these models lead to noise sensitive bursting which can be identified as a dynamical mechanism for almost periodic outbreaks of the herbivore infestation.

\indent The rest of paper is organized as follows. In Section 2 we define two classes of monotone dynamics of single plant species. In Section 3 we formulate general plant-herbivore models for the plant dynamics introduced in Section 2. In Section 4 we analyze the dynamic behavior of these two general models, e.g. the global stability of the boundary equilibrium and {uniform persistence} of these models. In Section 5 we apply the theoretical results from Section 4 to a Beverton-Holt model and a Holling-Type III model. The analysis and numerical simulations suggest that Beverton-Holt model goes through Neimark-Sacker bifurcation with unique periodic orbit for a certain set of {parameters} values; while Holling-Type III model goes through heteroclinic bifurcation for a certain set of {parameters} values. Our study also shows that noise is an important factor for outbreak of herbivore. Finally, we compare monotone plant growth models to unimodal and multimodal plant growth models regarding their influence of {plant-herbivore} dynamics.
%%%%%%%%%%%%%%%%%%%%%%%%%%%%%%%%%%%%%%%%%%
%%%%%%%%%%%%%%%%%%%%%%%%%%%%%%%%%%%%%%%%%%%%%
\section{Monotone growth dynamics for a single plant species}
Consider
\be
\label{plant1}
P_{t+1}=P_{t} f(r, P_t)=F(r, P_t), \hspace{2pt}t \geq 0.
\ee
where $P_t$ is the density of biomass in plant at generation $t$; {$F(r, P_t)$ is the growth function of biomass density} and $f(r, P_t)$ is the per capita growth rate of the biomass density. Without intra-specific competition, we have $f(r,0)=r$, i.e., $r$ is the maximal per capita growth rate of the plant. This simple formulation (\ref{plant1}) can give rise to a great diversity of dynamical behavior, depending on the expression used for the growth function $f(r,\cdot)$ and the values given to the parameters of that function. Several different functions have been considered. See \cite{Cohen1995} for a partial list of models with per capita growth rates that decline with increasing population density:
\be
\label{plant2}
{\frac{\partial f(r,P)}{\partial P} < 0, \hspace{2pt} P \geq 0.}
\ee
\noindent In biological terms, this means that the per capita growth rate {$f(r,P)$} decreases due to negative density dependent mechanism such as intra-specific competition between individuals within a population. {For convenience, we use $F(P), f(P)$ instead of $F(r, P), f(r,P)$ since $r$ is a fixed parameter.} The well known prototypes of the model (\ref{plant1}) under this biological assumption are the Beverton-Holt and Ricker models. The dynamics of Ricker model has been extensively studied (e.g., \cite{Ricker1954}; \cite{Kang2008}; \cite{Kon2006}). Here, we focus on the Beverton-Holt prototype, i.e., the dynamics of the plant is monotonically increasing,
\be
\label{plant3}
{F'(P)=\frac{d\,F(P)}{d\,P} \geq 0, \hspace{2pt} P \geq 0.}
\ee
{We can characterize the growth models of a single plant with assumption \textbf{H1} or  \textbf{H2} or both \textbf{H1} and \textbf{H2}: }
\begin{description}
\item{\textbf{H1}: $F(0)=0, \hspace{2pt} F(P)|_{P>0}> 0, \hspace{2pt} {F'(P)} >0,  \hspace{2pt}\mbox{and } \lim_{P \rightarrow +\infty} F(P) = C>0$.}%\\\vspace{3pt}
\item{\textbf{H2}: $f(P)|_{P\geq 0}\geq 0, \hspace{2pt} {f'(P)} < 0, \hspace{2pt} \mbox{and }\lim_{P \rightarrow +\infty} f(P) = 0$}.%\\\vspace{3pt}
\end{description}

\noindent In the biological sense, \textbf{H1} implies that the population density in one year is a increasing function {$F(P)$} of the density in the previous year and its per capita growth function {$f(P)$} may be increasing or decreasing or both, which implies that plant suffers from the extremes of contest intraspecific competitive interaction (Henson and Cushing \cite{Henson1996}); \textbf{H2} implies that the per capita growth function of the plant is a decreasing function due to intra-specific competition and the population density of a plant can be an increasing or decreasing function or both with respect to its density, which implies that plant suffers from the extremes of scramble intraspecific competitive interactions (Henson and Cushing \cite{Henson1996}). {In this article, we study the population dynamics associated with plants that satisfy  \textbf{H1} or \textbf{H2} or both \textbf{H1} and \textbf{H2}. The specific assumptions will be addressed in the models.} The following proposition summarizes the dynamics of plant in the absence of herbivore.
\begin{proposition}\label{p_single}
\begin{enumerate}
\item Assume that \textbf{H1} holds and there are $n+1$ consecutive, distinct and non-degenerate solutions ${\bar{P}^i}, i=0,1,...,n$ of $P = F(P)$ with the following property {$$0=\bar{P}^0<\bar{P}^1<\cdot\cdot\cdot<\bar{P}^{n}.$$} If ${\bar{P}^0}$ is stable (unstable) then the even ${\bar{P}^{i}}$ are stable (unstable) while the odd ${\bar{P}^{i}}$ are unstable (stable). In particular, ${\bar{P}^n}$ is always stable. Moreover, define the map $P_{t+1}=F(P_t)$, then for any $\epsilon>0$, there exists $N$ large enough, such that for all $t>N$, we have $$P_{t+1}=F(P_t)\leq {\bar{P}^n} + \epsilon.$$
\item Assume that  \textbf{H2} holds, then $P = P f(P)$ has at most two roots, i.e., $P=0$ and {the possible root of $1 = f(P)$}.
\end{enumerate}
\end{proposition}
\begin{proof}:
Possible configurations of the staircase diagrams of Fig. \ref{sink_source_fig} shows the alternating stable and unstable equilibria. {If  \textbf{H1} holds, then we have $$0<\lim_{P \rightarrow +\infty} F(P) = C<1\mbox{ and }0<{F'(\bar{P}^n)}<1.$$ This implies that the largest equilibrium $\bar{P}^n$ is locally stable.}
\begin{figure}[ht]
\centering
\includegraphics[width=0.8\textwidth]{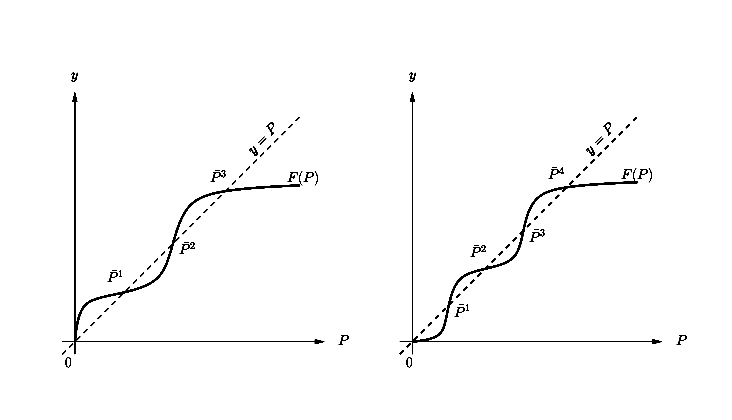}
\caption{Possible configurations of the staircase diagrams}
\label{sink_source_fig}
\end{figure}

\indent {If an initial condition $P_0$ satisfies $P_0\leq {\bar{P}^n}$, then \textbf{H1} implies that $$P_1=F(P_0)\leq F(\bar{P}^n)=\bar{P}^n$$ and by induction, $P_t\leq \bar{P}^n$ for all $t\geq 1$. In the case that the initial condition is larger than $\bar{P}^n$, i.e., $P_0> \bar{P}^n$, then \textbf{H1} and the fact that $\bar{P}^n$ is the largest positive root of $P=F(P)$ indicate that $F(P)<P$ for all $P>\bar{P}^n$. Thus, we have follows
$$\bar{P}^n=F(\bar{P}^n)\leq P_1=F(P_0)<P_0 .$$ Therefore, by induction, we know that the sequence $\{P_t\}_{t=0}^\infty$ is decreasing and converges to $\bar{P}^n$ as $t\rightarrow \infty$. This indicates that for any $\epsilon>0$, there exists $N$ large enough, such that for all $t>N$, we have $P_{t+1}=F(P_t)\leq \bar{P}^n + \epsilon$. In other cases, we have $P_{t+1}=F(P_t)< {\bar{P}^n}$ for all $t>0$.} 
%Assume $\bar{P}_i$ is a stable equilibrium for the map $P_{t+1}=F(r, P_t)$.

\indent Since $f(P)$ is a differentiable and strictly decreasing function of $P$, thus $1= f(P)$ has at most one solution. Therefore, the statement holds.
\end{proof}

%%%%%%%%%%%%%%%%%%%%%%%%%%%%%%%%%%%%%%%%%%%%%%%
%%%%%%%%%%%%%%%%%%%%%%%%%%%%%%%%%%%%%%%%%%%%%%%%%
\section{Plant-herbivore models}
Insect and plant survival rates often appear to be non-linear functions of plant and insect density,
respectively (\cite{Harper1977}; \cite{Crawley1983}). In our discrete-time models, we therefore assume that the plant
population growth is a non-linear function of herbivore and plant density, and that plant population growth decreases
gradually with increasing herbivore {density}. Similarly, we assume that the density of herbivore population
depends on both the plant and herbivore's density rather than only the herbivore density \cite{Crawley1983}. A final key feature of many plant-herbivore interactions is that, in the absence of the herbivore, we have a monotone growth dynamics as discussed in the previous section.

\indent Let $P_t$ represent the density of edible plant biomass in generation $t $ and $H_t$ represent the population density of herbivore.
The effect of the herbivore on the plant population growth rate is described by the function
$g(a,H_t)$ with $g(a, 0)=1$. Here the parameter $a$ measures the damage caused by herbivore, e.g., feeding rate.
We assume that the herbivore population density is proportional to a function of plant density $h(P_t)$ and a non-linear function of herbivore density $l(H_t)$. Therefore, the structure of our models is:
\bae\label{mon_ph1}
P_{t+1} &=& P_t f(P_t) g(a,H_t) \\
\label{mon_ph2}
H_{t+1} &=& h(P_t) l(H_t)
\eae
Many consumer-resource models assume a non-linear relationship between resource population size and
attack rate (\cite{Beddington1975}; \cite{Tang2002}). For plants and insect herbivores, we similarly expect a non-linear functional relationship, due to herbivore foraging time and satiation. The relationship is expressed in terms of plant biomass units rather than population size, because herbivores are unlikely to kill entire plants (\cite{Harper1977}; \cite{Crawley1983}).

\indent Our model has the following features: Without the herbivore, {we assume a monotone growth rate, i.e., {\bf H1}} holds. The growth function $F(P_t)$ determines the amount of new leaves available for consumption for the herbivore in generation $t$. We assume that the herbivores search for plants randomly. The area consumed is measured by the parameter $a$, i.e., $a$ is a constant that correlates to the total amount of the biomass that an herbivore consumes. The herbivore has a one year life cycle, the larger $a$, the faster the feeding rate. {After attacks by herbivores, the biomass in the plant population is reduced to 
\bae\label{ph4}
P_{t+1} &=& P_t f(P_t) e^{-a H_t}
\eae where $g(a, H_t)$ in \eqref{mon_ph1} is defined as
\bae\label{ph3}
g(a,H_t) &=& e^{-a H_t}.
\eae}  %a fraction of that present in the absence of herbivores.}

\indent The term $h(P_t)$ {in \eqref{mon_ph2}} describes how the biomass in the plants is converted to the biomass of the herbivore. It differs depending on the relative timing of herbivore feeding and growth. If the herbivore
attacks the plant before the plant grows, then we have
$h(P_t) = P_t$, otherwise, $h(P_t) = P_t f(P_t)$.
Since the biomass of herbivore comes from whatever they eat, $h(P_t)$ is the available biomass of a plant that can be
converted into the herbivore's biomass. The term $l(H_t)$ describes the fraction of $h(P_t)$ that can be used by the herbivore, i.e.,
$l(H_t) =1- e^{-a H_t}$.
Therefore, the evolution of the plant-herbivore system is either described by
\textbf{Model I :}
\bae\label{ph10}
P_{t+1} &=& F(P_t)e^{-a H_t} \\
\label{ph11}
H_{t+1} &=& P_t  \left[1- e^{-aH_t}\right]
\eae
describing the dynamics of a system where the plant is attacked before it has a chance to grow while
\textbf{Model II :}\\
\bae\label{ph12}
P_{t+1} &=& F(P_t)e^{-a H_t} \\
\label{ph13}
H_{t+1} &=& F(P_t)\left[1-e^{-a H_t}\right]
\eae
describes the dynamics when the plant grows first before being attacked.

%%%%%%%%%%%%%%%%%%%%%%%%%%%%%%%%%%%%%%%%%%%%%%%%%%%
%%%%%%%%%%%%%%%%%%%%%%%%%%%%%%%%%%%%%%%%%%%%%%%%%%%
\section{Mathematical analysis}
First, we can easily see that {$\mathbb R^2_+$ is positively invariant for both Model I and II}. In addition,
\begin{lemma}\label{bounded}
If \textbf{H1} holds, then $\limsup_{t\rightarrow \infty}\max\{P_t, H_t\}\leq \bar{P}^n$ for both Model I and II.
\end{lemma}
\begin{proof}:
For Model I,
\[H_{t+1} = P_t  \left[1- e^{-aH_t}\right] \leq P_t\]
For Model II,
\[H_{t+1} = P_t f(P_t) \left[1- e^{-aH_t}\right] \leq F(P_t)\]
Since {Condition \textbf{H1} holds for $F(P)$}, then from Proposition (\ref{p_single}), we can conclude that for any $\epsilon >0$, there exists $N$ large enough, such that for all $t>N$, the following holds
\[P_{t+1} = F(P_t) e^{- a H_t} \leq F(P_t) \leq \bar{P}^n + \epsilon .\] Therefore, we have $\limsup_{t\rightarrow \infty}\max\{P_t, H_t\}\leq \bar{P}^n$ for both Model I and II.
\end{proof}

%%%%%%%%%%%%%%%%%%%%%%%%%%%%%%%%%%%%%%%%%%%%%%%%%%%
\subsection{Equilibria and their stability}
If, in the absence of the herbivore there exist $n+1$ equilibria of the plant dynamics, then both Model I and II have $n+1$ boundary equilibria
of the form
\[E_{00}=(0, 0) \mbox{ and } {E_{i0}}= (\bar{P}^i, 0), i=1,2,..,n.\]
Their local stability can be determined by the eigenvalues of their Jacobian matrices.

%\begin{enumerate}
%\item Model I:
%\bae\label{J1}
%J_1&=& \left[ \begin {array}{cc} \frac{\partial F}{\partial P} e^{-a H}& -a P\\\noalign{\medskip}1- e^{-a H}& a P\end {array} \right]\Bigg\vert_{{P=\bar{P}^i}}
%\eae
%\item Model II:
%\bae\label{J2}
%J_2&=& \left[ \begin {array}{cc} \frac{\partial F}{\partial P} e^{-a H}& -a P\\\noalign{\medskip} \frac{\partial F}{\partial P} \left[1- e^{-a H}\right]& a P\end {array} \right]\Bigg\vert_{{P=\bar{P}^i}}
%\eae
%\end{enumerate}
{It is easy to check that the Jacobian matrices of Model I and II are identical at these boundary equilibria: the eigenvalues of their associated Jacobian matrix at $(0,0)$ are {$F'(0)$} and $0$; the eigenvalues of their associated Jacobian matrix at $(\bar{P}^i,0)$ are {$F'(\bar{P}^i)$} and $a \bar{P}^i$.}
The following theorems summarize the global dynamics:
\begin{theorem}\label{1th2}
{Assume that \textbf{H1} holds for both Model I and II. If {$F'(0)<1$} and $(0,0)$ is the only boundary equilibrium, then Model I and II are globally stable at $(\bar{P}^0,0)=(0,0)$. More generally, if $a{\bar{P}^n}<1, n\in \mathbb Z_+$, then $\lim_{t\rightarrow \infty} H_t =0$ for both Model I and II.}
\end{theorem}
\begin{proof}
From Lemma (\ref{bounded}), we know that for any $\epsilon >0$, there exists $N$ large enough, such that for all $t>N$, we have
\bae
\label{bound_P}
P_{t+1} &=& F(P_t) e^{- a H_t} \leq F(P_t) \leq {\bar{P}^n} + \epsilon
\eae
Since $a {\bar{P}^n}<1$, then for $\epsilon$ small enough, we have 
$$a P_t \leq a ({\bar{P}^n} + \epsilon) <1\mbox{ and } a F(P_t)\leq a({\bar{P}^n} + \epsilon)<1\mbox{ for all } t\geq N.$$
Thus, for Model I,
\bae
\label{H1_0}
H_{t+1} &=& P_t  \left[1- e^{-aH_t}\right] = H_t P_t \frac{\left[1- e^{-aH_t}\right]}{H_t} \leq a H_t P_t \leq a ({\bar{P}^n} + \epsilon)H_t
\eae
and for Model II,
\bae
\label{H2_0}
H_{t+1} &=& F(P_t) H_t \frac{\left[1- e^{-aH_t}\right]}{H_t} \leq a H_t F(P_t)\leq a ({\bar{P}^n} + \epsilon)H_t
\eae
Therefore, we have $H_t\leq [a ({\bar{P}^n} + \epsilon)]^{t-N}H_N, \mbox{ {for all} } t>N.$ This indicates that $\lim_{t\rightarrow \infty} H_t =0 $ for both Model I and II. Hence {solutions of} Model I and II are globally attracted to the boundary dynamics.
\end{proof}
\indent Theorem \ref{1th2} indicates that herbivore can not maintain its population if its attracting rate is too small or there is no enough food, i.e., $a{\bar{P}^n} <1$. In addition, {the special case of Theorem \ref{1th2} when $n=1$ leads to the following remarks}.\\\\
\noindent \textbf{Remark:} {Assume that the hypotheses of Theorem 4.2 hold. If $n=1$,  then from Proposition (\ref{p_single})}, we have
\begin{enumerate}
\item If {$F'(0) <1$}, then $\bar{P}^1$ is a source;
\item If {$F'(0) >1$}, then $\bar{P}^1$ is a sink.
\end{enumerate}
Hence, if  {$F'(0) >1$} and $n=1$, then $(\bar{P}^1,0)$ attracts all nontrivial solutions.
%Remark: For $n>1$ and $a {\bar{P}^n} <1$, we still have $\lim_{t\rightarrow \infty} H_t =0$ for  Model I and II. However, for $n>1$ there are multiple stable equilibria on the boundary.
%%%%%%%%%%%%%%%%%%%%%%%%%%%%%%%%%%%%%%
%%%%%%%%%%%%%%%%%%%%%%%%%%%%%%
\subsection{Unique interior equilibrium}
Interior equilibria are determined by the intersections of the nullclines. {Notice that if $\textbf{H1}$ holds, then $y=F(P)$ is a differentiable and monotone function of $P$ and maps $\mathbb R^{+}$ to $[0, C)$. Its inverse exists and can be written as $P = F^{-1}( y)$ which maps $[0, C)$ to $\mathbb R^{+}$. Similarly, if \textbf{H2} holds, then $y = f(P)$ is a differentiable and monotone function of $P$ and maps $\mathbb R^{+}$ to $[0, M)$. Its inverse exists and can be written as $P = f^{-1}( y)$ which maps $[0, M)$ to $\mathbb R^{+}$. Here $C=F(\infty)$ and $M=f(0)$ are some positive constants.}
If $(\bar{P}, \bar{H})$ is an interior equilibrium, then it is the solution of the two equations
\begin{enumerate}
\item For Model I:
\bae\label{n11}
P &=& f^{-1}(e^{a H})\\
 \label{n12}
 P &=& \frac{H}{1- e^{-a H}}.
\eae
\item For Model II:
\bae
 \label{n22}
 P &=& \frac{H}{e^{a H} - 1}\\
  \label{n23}
 P &=& F^{-1}(\frac{H}{1- e^{-a H}}).
 \eae
\end{enumerate}
If $F(P)$ is monotonically increasing, i.e., \textbf{H1} holds, then $F^{-1}(\frac{H}{1 - e^{-a H}})$ is an increasing function of $H$ which attains its minimum at $H=0$, i.e., $$\min_{H\geq 0}\big\{F^{-1}\left(\frac{H}{1 - e^{-a H}}\right)\big\}=F^{-1}(\frac{H}{1 - e^{-a H}})\big\vert_{H=0}=F^{-1}(\frac{1}{a}).$$
Similarly, if $f(P)$ is monotonically decreasing, i.e., \textbf{H2} holds, then $f^{-1}(e^{a H})$ is a decreasing function of $H$ which attains its maximum at $H=0$, i.e., $$\max_{H\geq 0}\big\{f^{-1}(e^{-a H})\big\}=f^{-1}(e^{-a H})\big\vert_{H=0}=f^{-1}(0).$$
\begin{theorem}\label{intercept1}
\begin{itemize}
\item[a)]  { Assume that both $\textbf{H1}$ and $\textbf{H2}$ hold for Model I, then Model I has at most one interior equilibrium which occurs when $f^{-1}(0) > \frac{1}{a}$.} The interior equilibrium emerges generically through a transcritical bifurcation from the largest boundary equilibrium ${\bar{P}^n}$ when ${\bar{P}^n} =\frac{1}{a}$, where $n\geq 1$.
\item[b)] {Assume that $\textbf{H1}$ holds for Model II, then Model II has at most one interior equilibrium which occurs when $F^{-1}(\frac{1}{a}) < \frac{1}{a}$.} The interior equilibrium emerges generically through a transcritical bifurcation from the largest
boundary equilibrium ${\bar{P}^n}$ when ${\bar{P}^n} = \frac{1}{a}$, where $n\geq 1$.
\end{itemize}
\end{theorem}
\begin{proof}:
The proofs for a) and b) are similar. We show case b):
The interior equilibria of Model II are determined by the intersections of the nullclines (\ref{n22}) and (\ref{n23}). Since (\ref{n22}) is a decreasing function and (\ref{n23}) is an increasing function, they have only one interior intersection if the following inequality holds
$$\min_{H\geq 0}\big\{F^{-1}\left(\frac{H}{1 - e^{-a H}}\right)\big\}<\max_{H\geq 0}\big\{\frac{H}{e^{aH}-1}\big\}\Rightarrow F^{-1}(\frac{1}{a}) < \frac{1}{a}.$$
 %$F^{-1}(P) = P$ has the same equilibrium solutions as $P= F(P)$. 

{The Jacobian matrix of Model II evaluated at the boundary equilibrium $({\bar{P}^i},0)$ is
\bae\label{J2_matrix}
J|_{({\bar{P}^i},0)}&=& \left[ \begin {array}{cc} {F'(\bar{P}^i)}& -a {\bar{P}^i}\\\noalign{\medskip} 0& a {\bar{P}^i}\end {array} \right]
\eae with its eigenvalues as $\lambda_1=a{\bar{P}^i}$ and $\lambda_2={F'(\bar{P}^i)}$. Thus, at the largest boundary equilibrium $({\bar{P}^n},0)$, we have}
\begin{enumerate}
 \item $\lambda_1|_{({\bar{P}^n},0)}=a{\bar{P}^n}$ and $\lambda_2|_{({\bar{P}^n},0)}={F'(\bar{P}^n)}<1$;
 \item $\frac{\partial \lambda_1}{\partial a}|_{({\bar{P}^n},0)}={\bar{P}^n}$ and $\frac{\partial \lambda_2}{\partial a}|_{({\bar{P}^n},0)}=0$
 \end{enumerate}
In the case that ${\bar{P}^n}=\frac{1}{a}$, we have $\lambda_1|_{({\bar{P}^i},0)}=a{\bar{P}^i}<1, i=1,...,n-1$ and 
$$\lambda_2|_{({\bar{P}^n},0)}=a{\bar{P}^n}<\lambda_1\vert_{({\bar{P}^n},0)}=1\mbox{ and }\frac{\partial \lambda_1}{\partial a}\vert_{({\bar{P}^n},0)}={\bar{P}^n}=\frac{1}{a}>0.$$The eigenvector associated with the eigenvalue $\lambda_1\vert_{({\bar{P}^n},0)}=a {\bar{P}^n}$ is
\bae\label{J2_E1v}
V|_{\lambda_1=a {\bar{P}^n}}&=& \left[ \begin {array}{c} -\frac{a {\bar{P}^n}-{F'(\bar{P}^n)}}{a {\bar{P}^n}}x_2\\\noalign{\medskip} x_2\end {array} \right]
\eae
If $a{\bar{P}^n}=1$, then the two components of (\ref{J2_E1v}) {have opposite signs}. This implies that by choosing $x_2>0$, the unstable manifold of $E_{n0}$ points toward the interior of $X_{11}$. Therefore, apply Theorem 13.5 in the book by Smoller \cite{Smoller1994}, the unique interior equilibrium of Model II emerges generically through a transcritical bifurcation from the largest boundary equilibrium ${\bar{P}^n}$ when ${\bar{P}^n} = \frac{1}{a}$, where $n\geq 1$.
\end{proof}
%\subsection{Permanence}
\subsection{{Uniform persistence} of Model I and II}
We define the sets 
\bea
X&=&\{(P,H): P\geq 0, H\geq 0\}\\
X_{11}&=&\{{(P,H)\in X: P H> 0}\}\\
\partial X_{11}&=&X\setminus X_{11}
\eea
and consider the additional hypothesis
\begin{description}
\item\textbf{H3}:\, The smallest positive root ${\bar{P}^1}$ of $P=F(P)=Pf(P)$ satisfies $a\bar{P}^1>1$, and in addition, $f(0)>1$. 
\end{description}
In the following, we show that Model I and II  are uniformly persistent with respect to $(X_{11},\partial X_{11})$ {if both \textbf{H1} and \textbf{H3} hold}, i.e., {for any initial condition $(P_0, H_0)\in X_{11}$}, there exists some $\epsilon>0$ such that $\liminf_{t\rightarrow \infty}\min\{P_t, H_t\}\ge \epsilon$.

\begin{lemma}\label{l_bh_invariant}
 $X_{11}$ and $\partial X_{11}$ are positively invariant for (\ref{ph10})-(\ref{ph11}) and (\ref{ph12})-(\ref{ph13}).
\end{lemma}
\noindent The following theorem is the main result of this subsection.
\begin{theorem}\label{persistent}
 If $a\bar{P}_1>1$, then (\ref{ph10})-(\ref{ph11}) and (\ref{ph12})-(\ref{ph13}) are uniformly persistent with respect to $(X_{11},\partial X_{11})$ provided that they satisfy both  \textbf{H1} and \textbf{H3}.
 \end{theorem}

\begin{proof}
From {Lemma \ref{l_bh_invariant} and Proposition \ref{bounded}}, we obtain that the systems (\ref{ph10})-(\ref{ph11}) and (\ref{ph12})-(\ref{ph13}) are point dissipative. {This combined with the fact that the semiflow generated by these systems is asymptotically smooth (this is automatic, since the state space is in ${\mathbb R^2_+}$), gives the existence of the compact attractors of points for both systems ({Smith and Thieme 2010} \cite{Smith_Thieme2005}).}

\indent Notice that the omega limit set of $S_1=\{(P,H)\in {\mathbb R^2_+}: P=0\}$ is the trivial boundary equilibrium ${E_{00}}$. Let $L(P,H)=P$ be an average Lyapunov function, then we have $L(P,H)\vert_{S_1}=0$. Since the system satisfies \textbf{H3}, then the following inequality holds
 {$$\sup_{t\geq 0}\liminf_{(P_0,H_0)\rightarrow (0,0)}\frac{P_t}{P_0}=\sup_{t\geq 0}\liminf_{(P_0,H_0)\rightarrow (0,0)}\prod_{j=0}^{t-1}f(P_j)e^{-a H_j}=\sup_{t\geq 0} \left(f(0)\right)^t>1$$ }where $(P_0,H_0)\in X\setminus S_1$. Therefore, by applying {Theorem 2.2 in \cite{Hutson1984}} and its corollary to the systems (\ref{ph10})-(\ref{ph11}) and (\ref{ph12})-(\ref{ph13}), we obtain persistence of the plant population, i.e., for any initial condition $P_0>0$, we have $\liminf_{t\rightarrow \infty}{{P_t}}\ge \epsilon.$
 
The fact that the plant population is uniformly persistent implies that the system (\ref{ph10})-(\ref{ph11}) or (\ref{ph12})-(\ref{ph13}) can be restricted in $X\cap\{(P,H)\in\mathbb R^2_+: P\geq \epsilon\}$. According to {Proposition \ref{p_single}}, we can conclude that the omega limit sets of $S_2=\{(P,H)\in \partial X_{11}: P>0\}$ are $\{E_{i0},1\leq i\leq n\}$. Since {$a\bar{P}^1>1$, Condition \textbf{H3}} indicates that ${a\bar{P}^i}>1,1\leq  i\leq n$. Now define $L(P,H)=H$ as an average Lyapunov function, then we have $L(P,H)\vert_{S_2}=0$. Moreover, {for the model (\ref{ph10})-(\ref{ph11}), we have
 {$$\sup_{t\geq 0}\liminf_{(P_0,H_0)\rightarrow (\bar{P}^i,0)}\frac{H_t}{H_0}=\sup_{t\geq 0}\liminf_{(P_0,H_0)\rightarrow (\bar{P}^i,0)}\left(\prod_{j=0}^{t-1}{P_j\frac{1-e^{aH_j}}{H_j}}\right)^t=\sup_{t\geq 0} (a \bar{P}^i)^t>1$$}
and for the model  (\ref{ph12})-(\ref{ph13}) we have
 {$$\sup_{t\geq 0}\liminf_{(P_0,H_0)\rightarrow (\bar{P}^i,0)}\frac{H_t}{H_0}=\sup_{t\geq 0}\liminf_{(P_0,H_0)\rightarrow (\bar{P}^i,0)}\left(\prod_{j=0}^{t-1}{F(P_j)\frac{1-e^{aH_j}}{H_j}}\right)^t=\sup_{t\geq 0} (a \bar{P}^i)^t>1$$}where $(P_0,H_0)\in X_{11}$}.
Therefore, according to Theorem 2.2 and its corollary 2.3 in \cite{Hutson1984}, we can show that the systems (\ref{ph10})-(\ref{ph11}) and (\ref{ph12})-(\ref{ph13}) are uniformly persistent. Hence, the statement holds.
\end{proof}
{\textbf{Remark:} The arguments used to prove Theorem \ref{persistent} are standard, which can be found in many literature (e.g., \cite{Hofbauer1987}, \cite{Kon2004},  \cite{Kang2010}).}

%%%%%%%%%%%%%%%%%%%%%%%%%%%%%%%%%%%%%%%%%%%
\section{Application and simulations}
\subsection{The Beverton-Holt and Holling-Type III models}
In this section, we focus on two typical models for the plant dynamics and apply our results:
\begin{enumerate}
\item \bae\label{hb1}
P_{t+1}&=& F(P_t) = \frac{r P_t}{1 + P_t}
\eae
is the Beverton-Holt model where $F(P_t)$ satisfies the assumptions of $\textbf{H1}$ and $f(P_t)$ satisfies those of $\textbf{H2}$.
The two equilibria are $\bar{P}^0=0$ and  $\bar{P}^1= r-1$.
From Proposition \ref{p_single}, we know that $\bar{P}^0$ is a sink if $r<1$; $\bar{P}^0$ is a source if $r>1$. In addition, $\bar{P}^1$, if it exists, is always a sink.
\item  A Hollings type III model is given by
\bae\label{hb2}
P_{t+1}&=& F(P_t) = \frac{r P_t^2}{1 + P_t^2}
\eae
where $F(P_t)$ satisfies $\textbf{H1}$. The equilibria are $ \bar{P}^0 =0, \bar{P}^1=\frac{r-\sqrt{r^2-4}}{2}$ and $\bar{P}^2= \frac{r+\sqrt{r^2-4}}{2}$. If $r<2$
then $\bar{P}^0$ is the only equilibrium and it is globally stable. If $r>2$; $\bar{P}^0$ is a sink, $\bar{P}^1$ is a source and $\bar{P}^2$ is a sink.
\end{enumerate}
The plant-herbivore models with (\ref{hb1}) and (\ref{hb2}) as plant dynamics become:\\
{\bf Model I:}
\begin{enumerate}
\item
\bae\label{hb11}
P_{t+1}&=& \frac{r P_t}{1 + P_t} e^{- a H_t}\\[2mm]
\label{hb12}
H_{t+1}&=& P_t \left[1- e^{- a H_t}\right]
\eae
\item  \bae\label{hb13}
P_{t+1}&=& \frac{r P_t^2}{1 + P_t^2} e^{- a H_t}\\[2mm]
\label{hb14}
H_{t+1}&=& P_t \left[1- e^{- a H_t}\right]
\eae
\end{enumerate}
{\bf Model II:}
\begin{enumerate}
\item \bae\label{hb21}
P_{t+1}&=& \frac{r P_t}{1 + P_t} e^{- a H_t}\\[2mm]
\label{hb22}
H_{t+1}&=& \frac{r P_t}{1 + P_t} \left[1- e^{- a H_t}\right]
\eae
\item  \bae\label{hb23}
P_{t+1}&=& \frac{r P_t^2}{1 + P_t^2} e^{- a H_t}\\[2mm]
\label{hb24}
H_{t+1}&=& \frac{r P_t^2}{1 + P_t^2}\left[1- e^{- a H_t}\right]
\eae
\end{enumerate}
Applying the results of the previous section, we have the following two corollaries.
\begin{corollary}\label{application}
The three models (\ref{hb11})-(\ref{hb12}),(\ref{hb21})-(\ref{hb22}) and (\ref{hb23})-(\ref{hb24}) have at most one interior equilibrium. The interior equilibria of Model I (\ref{hb11})-(\ref{hb12}) and Model II (\ref{hb21})-(\ref{hb22}) emerge through a transcritical bifurcations from the boundary equilibrium $(\bar{P}^1, 0)=(r-1, 0)$ when $a (r-1) = 1$ and $r>1$; The interior equilibrium of Model II (\ref{hb23})-(\ref{hb24}) emerges through a transcritical bifurcations from the boundary equilibrium $(\bar{P}^2, 0)=(\frac{r + \sqrt{r^2-4}}{2}, 0)$ when $a \frac{r + \sqrt{r^2-4}}{2} =1$ and $r>2$.
\end{corollary}

\begin{corollary}\label{application2}
If $a(r-1)>1$, the systems (\ref{hb11})-(\ref{hb12}) and (\ref{hb21})-(\ref{hb22}) are uniformly persistent with respect to $(X_{11},\partial X_{11})$. \end{corollary}
\subsection{Periodic orbits and heteroclinic bifurcations}
The stability of the single interior equilibrium of models (\ref{hb11}) to (\ref{hb24})
depends on the values of the parameters $r$ and $a$. As the values of $r$ or $a$ increase, the interior equilibrium
goes through a Neimark-Sacker bifurcation generating an invariant cycle.

Since Model I and Model II have similar dynamics, we only focus on Model II and discuss the Beverton-Holt model (\ref{hb1}) and the Holling-Type III model (\ref{hb2}), respectively. The main difference between the Beverton-Holt model and the Holling-Type III model is that the Holling-Type III model can show a heteroclinic bifurcation where a periodic orbit grows until it becomes a heteroclinic connection between boundary equilibria whereas the Beverton-Holt model does not show such a bifurcation. Figure (\ref{het_bif_fig}) shows the heteroclinic bifurcation schematically: When $a=0.71$ and $r=2.5$, the system has a stable interior equilibrium (the dark dot that is in the middle of the figure, which is generated by the Matlab); When we increase $r$ to 3.5 and keep $a=0.71$, the system has an invariant orbit (the grey orbit in the figure, which is generated by the Matlab); However, if we continue to increase the values of $a$ or $r$, the invariant orbit disappears and the system converges to the boundary equilibrium $(0,0)$. This suggests that a heteroclinic bifurcation occurs (the dark line with arrows in the figure, which is generated schematically).
%\begin{figure}[ht]
%\begin{center}
%\includegraphics[width=.6\textwidth]{bh2_lm_hc0.eps}
%\caption{The Heteroclinic Bifurcation of Holling-Type III model happens at $a=0.71, r=3.5$.}
%\label{het_bif_fig}
%\end{center}
%\end{figure}

\begin{figure}[ht]
\centering
\includegraphics[width=0.8\textwidth]{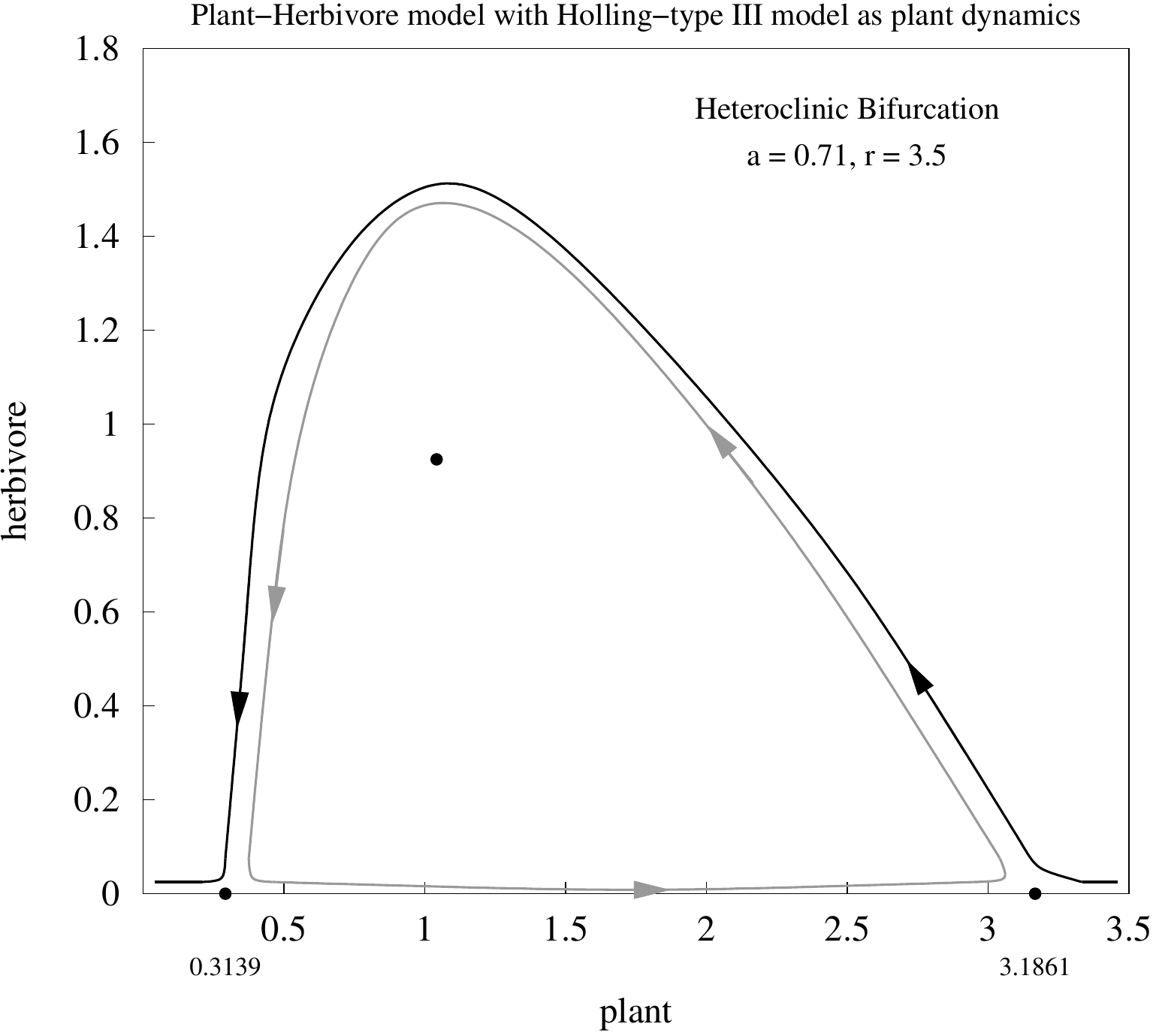}
\caption{Schematic of the heteroclinic bifurcation of Holling-Type III model happens at $a=0.71, r=3.5$.}
\label{het_bif_fig}
\end{figure}
Since the Beverton-Holt model only has the origin as a saddle and one other boundary equilibrium and since the stable manifold of the
origin is the $H$-axis which is an invariant manifold, the periodic orbit in the interior cannot become heteroclinic. However, it can become very large and
pass the origin arbitrarily close to the coordinate axes as shown in Fig. \ref{b-h_lc}. Figure \ref{b-h_lc} is the numerical simulations generated by the Matlab for 2000 generations when $a=2$ and $r=2.5,2.7,2.8,3$. When $a=2$ and $r=2.5$, the system has a stable interior equilibrium as shown in the figure (small dark dot); when $r=2,7, 2.8, 3$, the system has an invariant orbit.
\begin{figure}[ht]
\centering
\includegraphics[width=.6\textwidth]{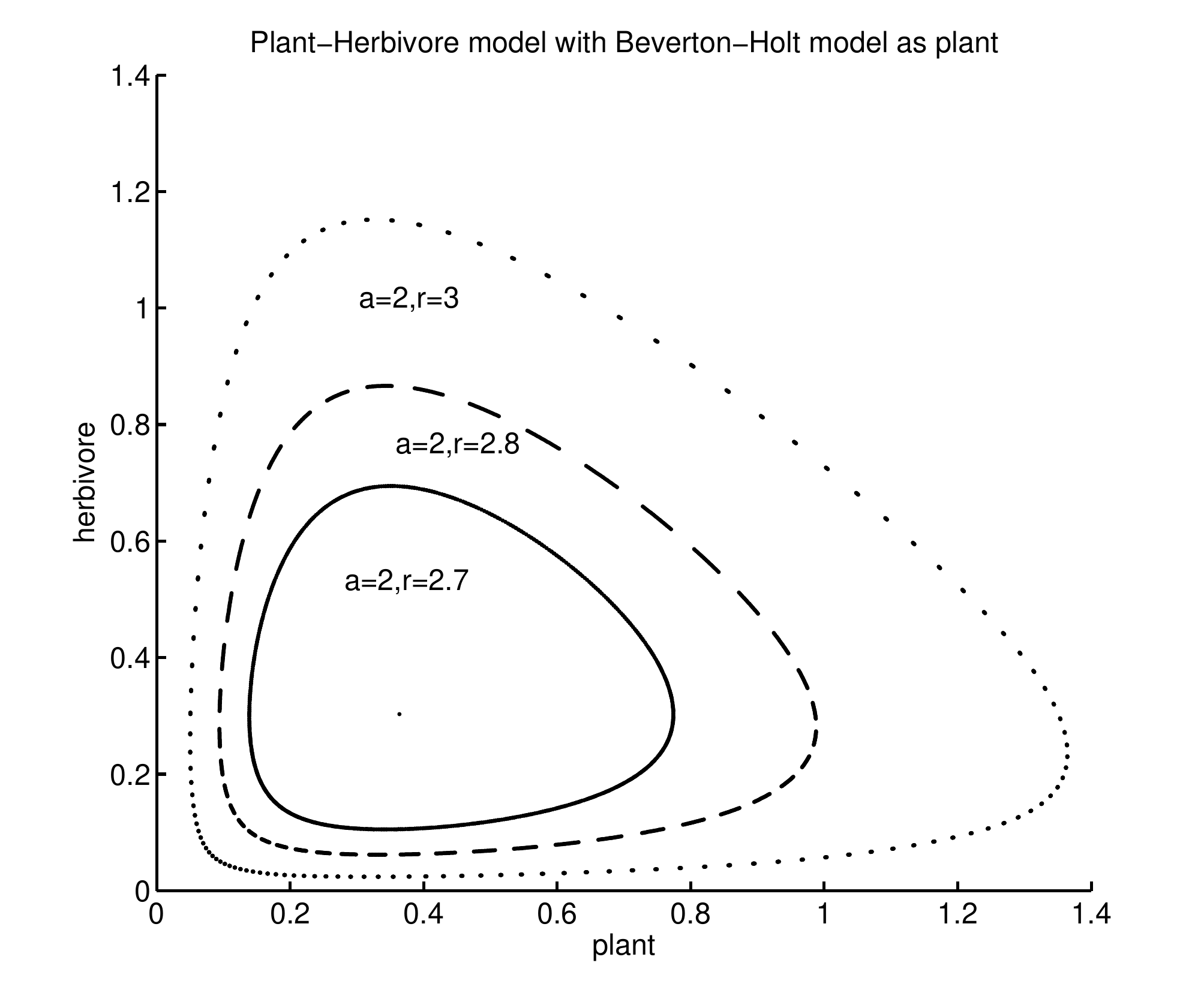}%it was bh1_lm
\caption{The periodic orbit for the Beverton-Holt model when $a=2, r=2.5, 2.7, 2.8, 3$}
\label{b-h_lc}
\end{figure}
Numerical simulations of this case hint at an interesting phenomenon: {a} standard numerical simulation shows the periodic orbit disappearing
and the trajectory approaching the nontrivial boundary equilibrium as time increases. However, that boundary equilibrium is a saddle and
the trajectory should leave into the interior but it does not do so over any simulation time that we checked. The resolution of the puzzle comes
from the considerations of the accuracy of the simulations: As the limit cycle gets closer to the origin the herbivore values become so small that they
are approximated as zero. Hence the dynamics is reduced to the dynamics of the plant which has a stable equilibrium on the invariant manifold determined by
$H=0$ and hence the trajectory never leaves.

Figure \ref{monotone1}(a) shows a ``bifurcation diagram" for the
Beverton-Holt model, which describes the Neimark-Sacker bifurcation curve (dashed line) and the ``collapse curve" (solid line). The later
represents an interpolation of numerical simulations with
$a$ and $r$ values for which a standard matlab numerical precision simulation does not detect a population of the herbivore.
Figure \ref{monotone1}(b) is a bifurcation diagram for a Holling-Type III model showing interpolations of the
Neimark-Sacker bifurcation curve (dashed line) and the heteroclinic bifurcation (solid line), respectively.

\begin{figure}[ht]
\begin{center}
\subfigure [The bifurcation diagram for the Beverton-Holt model.]
{\includegraphics[width=0.44\textwidth]{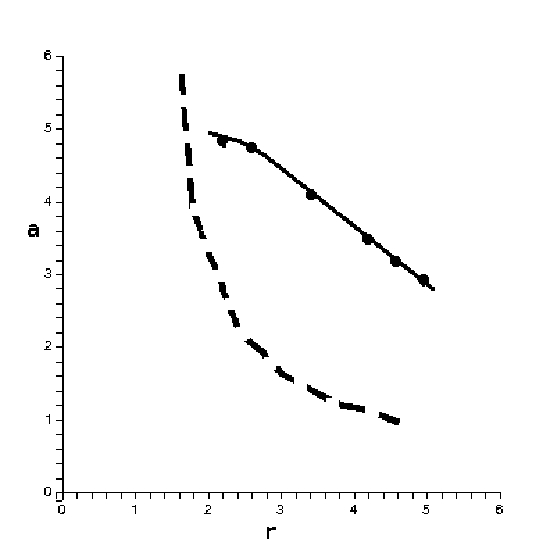}}\hspace{0.6cm}
%{\includegraphics[width=0.44\textwidth]{bh1_bif0.pdf}}\hspace{0.6cm}
\subfigure [The bifurcation diagram for the Holling-Type III model.]
{\includegraphics[width=0.44\textwidth]{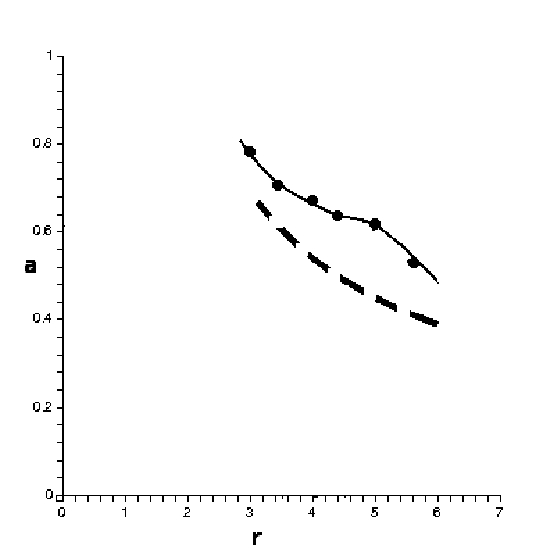}}\hspace{0.6cm}
%{\includegraphics[width=0.445\textwidth]{bh2_bif0.pdf}}
\caption{Neimark-Sacker bifurcation and heteroclinic bifurcations}
\label{monotone1}
\end{center}
\end{figure}
%%%%%%%%%%%%%%%%%%%%%%%%%%%%%%%%%%%%
\subsection{Noise-generated outbreaks}
The extreme sensitivity of the periodic orbit in the Beverton-Holt model suggests that
noise may play a much bigger role than previously discussed in the outbreaks of herbivore infestations.
Once the periodic orbit disappears due to accuracy issues, we can make it re-appear by adding small amount of noise to the simulation:
\begin{enumerate}
 \item Noise: we use positive white noise to make sure  the system stays positive, i.e., we sample from a normal distribution but discard any negative noise sample;
\item  Population of herbivores: for each generation, we add the noise to the herbivore, i.e.,
\bae\label{hnoise}
H_{t+1}&=&\frac{r P_t e^{-a H_t}}{1+P_t} + \omega R_n
\eae
where $R_n$ is a positive white noise as defined above and $\omega$ is the amplitude of the noise. See Fig. \ref{monotone_bht} for example, in this case, the amplitude of the positive  white noise is $\omega=0.01$.
\end{enumerate}
At that time the trajectories look like a randomly occurring bursting phenomenon that nevertheless has a well defined
average periodicity (see Fig. \ref{monotone_bht}).
\begin{figure}[ht]
\centering
\includegraphics[width=0.6\textwidth]{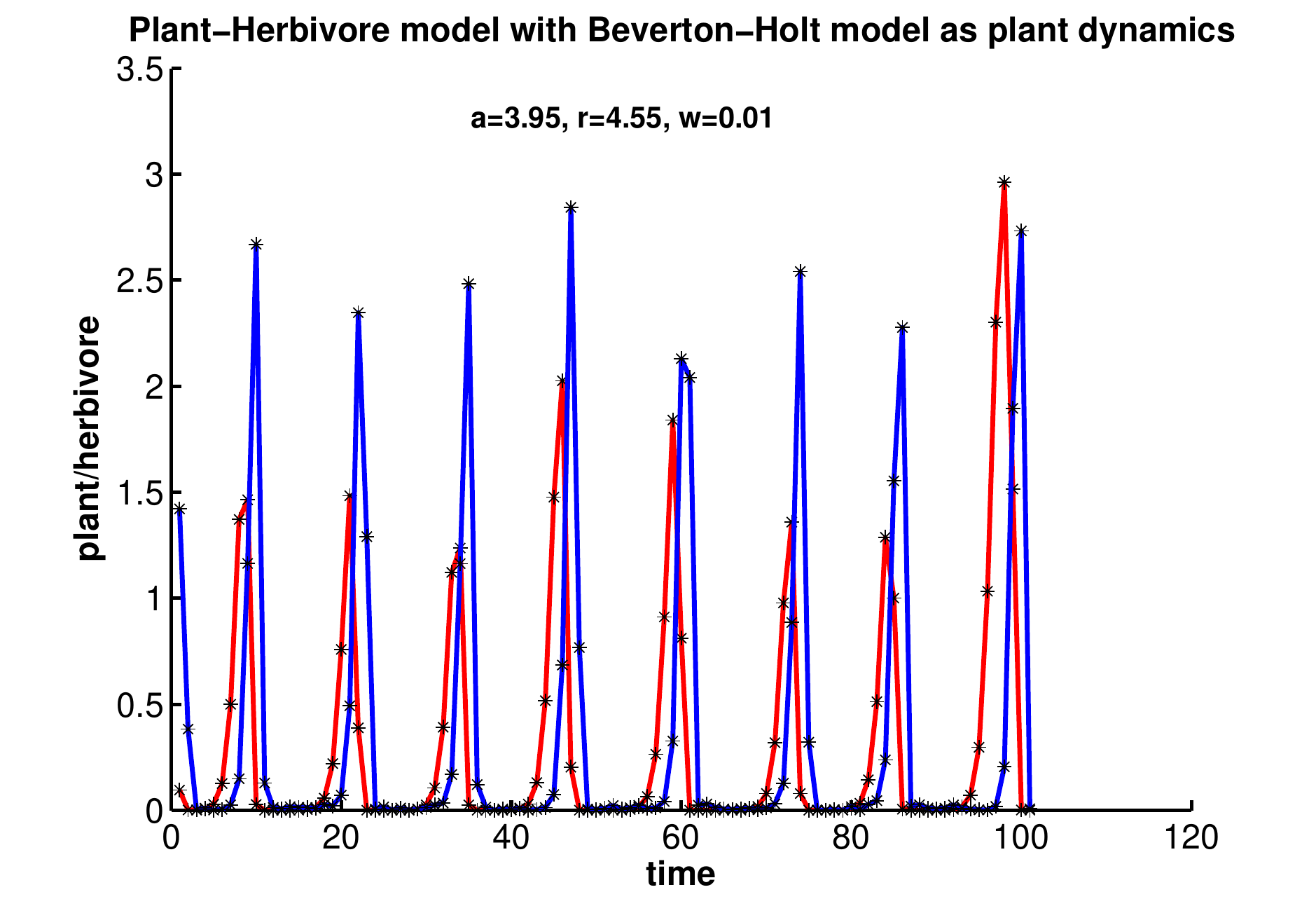}
\caption{Time series of the herbivore population for the Beverton-Holt model. The parameters are $a=3.95, r=4.55$ and a noise level of $w=0.01$}
\label{monotone_bht}
\end{figure}
Given by the exact nature of the model there will be a {\it threshold} at which the population of the herbivore cannot be detected in nature.
We define the {\it resident time} as the time interval for which the population of the herbivore stays below some threshold, e.g., 0.01 and
the {\it resident time ratio} as the ratio of the residence time to the period of the bursting.
Table \ref{t} shows the period as a function of the mean square amplitude of the noise level. Figure \ref{resident}
shows the resident time ratio as a function of the noise amplitudes. The figure is generated by calculating
the resident time ratio for each noise amplitude for 50 trajectory with 1000 generations.
The Figure shows that over many orders of magnitude the residence ratio stays around $80\%$ indicating that the
herbivore is dormant for most of the time and only appears for about $20\%$ of its periodic cycle.
The Table indicates that by choosing a particular noise level, we can control the apparent periodicity of the bursts.
\begin{figure}[ht]
\centering
\includegraphics[width=0.6\textwidth]{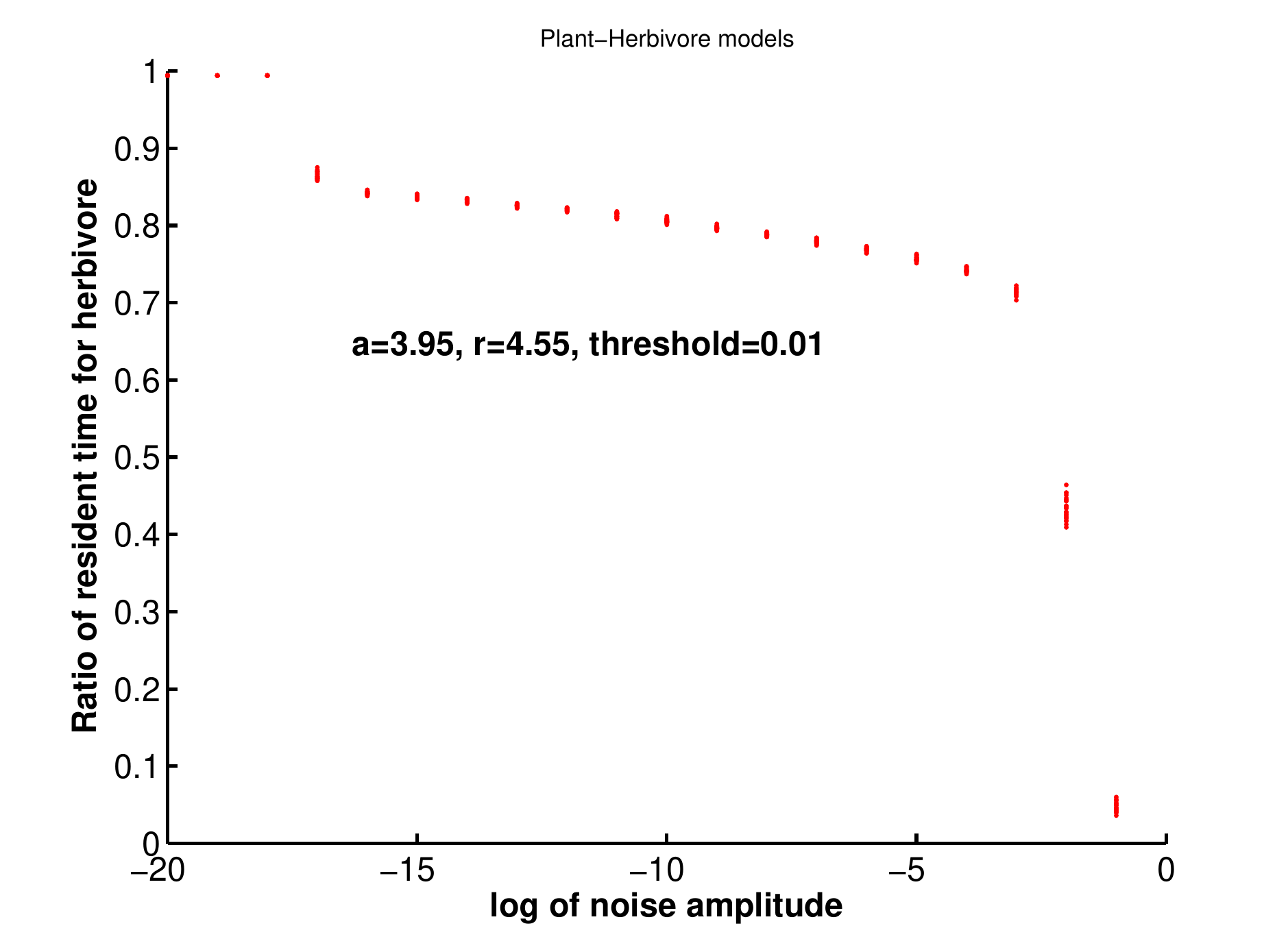}
\caption{The resident time ratio as a function of the noise amplitude when $a=3.95, r=4.55$ with a threshold of 0.01.}
\label{resident}
\end{figure}

\begin{table}[ht]\label{t}
\begin{center}
\begin{tabular}{|c|c|c|c|c|c|c|} \hline
Amplitude of Noise $w$ &0.01&0.001&0.0001&0.00001&0.000001&0.0000001\\
\hline
Period $t$& 8&10&12&14&17&19\\
\hline
\end{tabular}
\end{center}
\caption{Average period of the herbivore dynamics when $a=3.95, r=4.55$.}
\end{table}
{In particular,} time intervals of the herbivore outbreaks around $8-12$ years
can be {generated,}  which fits the ecological data for gypsy moth out-breaks \cite{Liebhold1996}.
Also, for larger noise levels, the distribution of the periods is {rather broad,}
which also seems to be happening for real data (\cite{Kendall1999}; \cite{Rinaldi2001}).

%%%%%%%%%%%%%%%%%%%%%%%%%%%%%%%%%%%%

\section{Conclusions and additional features}
For most plant species it is conceivable that there is density dependent regulation of its growth. However, very few plants show periodic or strongly chaotic variation of the plant density from generation to generation. Hence it is important to determine the influence of models of monotone growth dynamics on the plant-herbivore interaction model.
 We proved three key features of such interactions that are important for model building:
\begin{itemize}
\item All monotone growth models generate a unique interior equilibrium.
\item Monotone growth models with just one sustainable equilibrium for the plant population (e.g. the Beverton-Holt model) lead to noise sensitive bursting. This certainly happens for many plant-herbivore systems and the dynamical mechanism discussed here has not been noticed before in plant-herbivore systems (however, see \cite{Rinaldi2001}).
\item Model I and II have a uniformly persistent property if they satisfy both \textbf{H1} and \textbf{H3}. In particular, the Beverton-Holt model is uniformly persistent and the Holling-Type III model is not.
\item The Beverton-Holt model does not have more complicated dynamics than a periodic orbit in the interior of the phase space. Although we cannot prove this, we conjecture that this is true for all models that satisfy the assumptions of {\bf H1} and {\bf H2}, i.e. have just one equilibrium for the pure plant dynamics.
\end{itemize}
Without any claim to a complete analysis of all types of {models,}  we note a few additional features associated with monotone and non-monotone plant growth models.\\
\begin{enumerate}
\item {\bf Bistability:}
The paper \cite{Kang2008} study plant-herbivore systems of Model II type with a Ricker model for the pure plant dynamics, also known as the modified Nicholson-Bailey model. It is shown that for a large set of parameters the system exhibits bistability between complicated (possibly chaotic) dynamics in the interior of the phase space and equally complicated dynamics on the boundary (Fig. \ref{crisis_unimodal}(a)). Kon \cite{Kon2006} discusses a similar bistability phenomenon for Model I. Since unimodal maps are all topologically equivalent \cite{Guckenheimer1979}, we expect bistability to be a defining feature for plant-herbivore models with unimodal plant growth models.

In contrast, models that satisfy the assumptions of {\bf H1} and {\bf H3}, e.g., the Beverton-Holt model cannot show bistability since the global attractor either is a fixed point on the boundary or some set in the interior of the phase space. However, it is conceivable that models that do not satisfy {\bf H3}, e.g. Holling-Type III  models, show bistability between an interior attractor and a boundary equilibrium that is not the largest equilibrium for the pure plant population.
\item {\bf Crises of Interior Attractors:}
All models seem to show some sort of global attraction to the boundary dynamics, i.e. extinction of the parasite for large growth rates $r$:
\begin{itemize}
\item Unimodal models show a crisis type of bifurcation whereas the chaotic dynamics in the interior collapses and the system becomes globally attracted to the boundary dynamics \cite{Kang2008}.
For instance the interior strange attractor in Fig. \ref{crisis_unimodal} that exists for a growth parameter of $r = 3.8$ will grow and hit the stable manifold of the boundary attractor for $r= 3.85$.
\begin{figure}[ht]
\begin{center}
\includegraphics[width=0.45\textwidth]{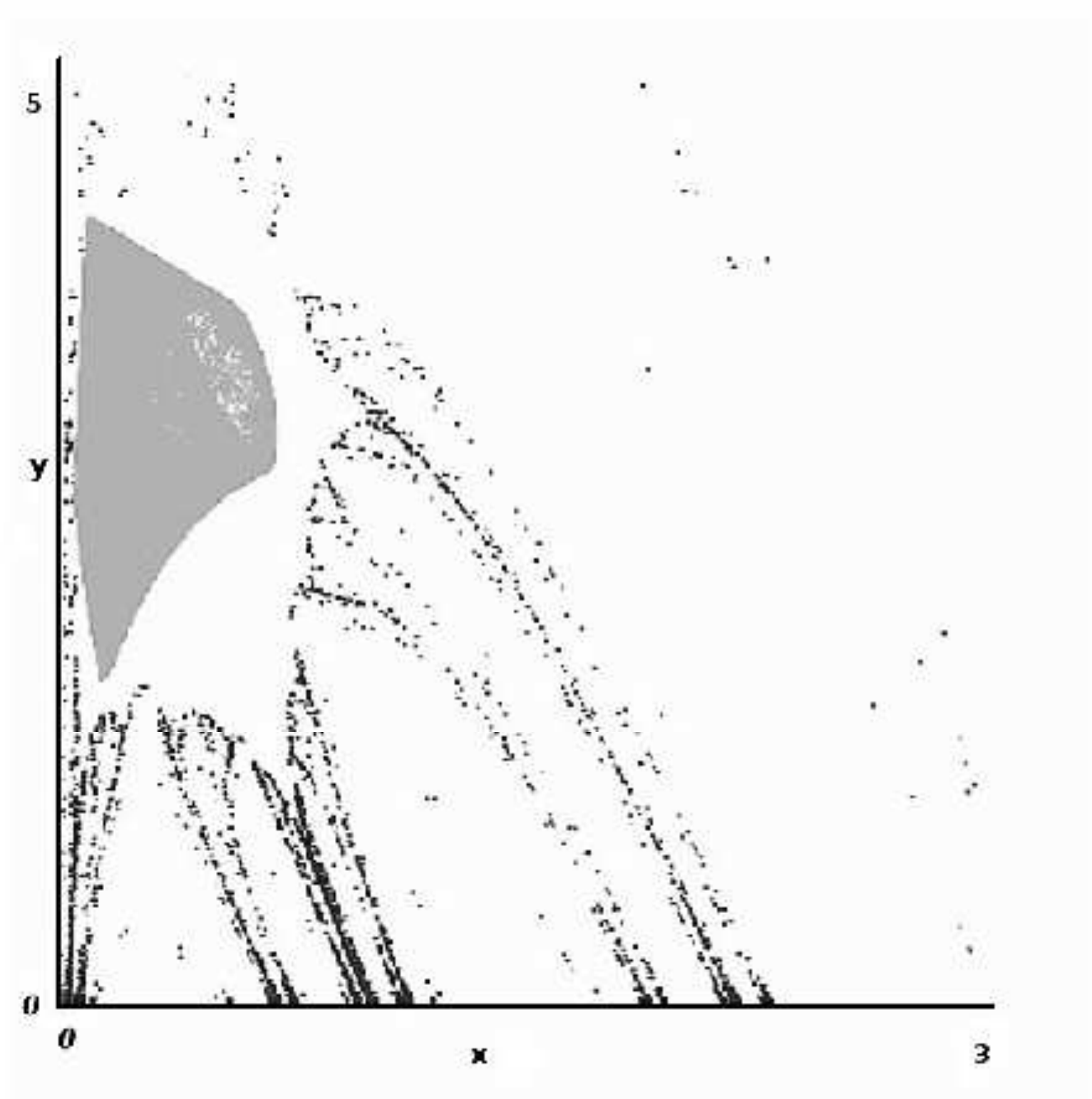}
\caption{The interior strange attractor and the stable manifold of the boundary attractor $a=0.95,r=3.8$.}
\label{crisis_unimodal}
\end{center}
\end{figure}
\item Holling-Type III models {show} a heteroclinic orbit which breaks and leads to global attraction to a boundary fixed point.
\item Conjecture: {The Beverton-Holt model does not lead to complete extinction according to Theorem (\ref{persistent}). However, $\epsilon$ in Theorem (\ref{persistent}) may happen to be small which, taking a stochastic effect into account, might lead to extinction of herbivores.}
 \end{itemize}
\end{enumerate}
%%%%%%%%%%%%%%%%%%%%%%%%%%%%%%%%%%%%%%%%%%
%%%%%%%%%%%%%%%%%%%%%%%%%%%%%%%%%%%%%%%%%%%%%
\section*{Acknowledgments}
The research of D.A. is supported by NSF grant DMS-0604986. The authors would like to thank the anonymous referees for comments that helped to improve the article. We also thank Nicolas Lanchier for his help producing some of the figures. 

%%%%%%%%%%%%%%%%%%%%%%%%%%%%%%%%%%%%%%%%%%%%%%
%\bibliography{ref_disnew2}

\bibliographystyle{plain}

%%%%%%%%%%%%%%%%%%%%%%%%%%%%%%%%%%%%%%%%%%%%
\label{lastpage}

\end{document}